\documentclass{jaac}
\def\currentvolume{*}
\def\currentissue{*}
\def\currentyear{*}
\def\currentmonth{*}
\def\doi{Not assigned}
\firstpage{1}

\usepackage{, amsmath, amssymb, amsfonts, amsthm, enumerate, color}
\usepackage{booktabs}
\usepackage[boxed,algosection]{algorithm2e}
\usepackage[american]{babel}
\newcommand{\mybox}[1]{%
	\begin{center}%
		\setlength{\fboxrule}{1pt}%
		\fbox{%
			\begin{minipage}{0.9\textwidth}%
				#1
			\end{minipage}%
		}%
	\end{center}%
}

\def\shorttitle{Accelerated projection-based forward-backward splitting algorithms for monotone inclusion problems}

\def\shortauthor{Bing Tan, Zheng Zhou   \& Xiaolong Qin }

\title{\uppercase{Accelerated projection-based forward-backward splitting algorithms for monotone inclusion problems}}

\author{ Bing Tan$^{1}$, Zheng Zhou$^{1}$, Xiaolong Qin$^{2,\dag}$}

\date{}

\begin{document}
\baselineskip 12pt

\maketitle
\begin{abstract}
In this paper, based on inertial and Tseng's ideas, we propose two projection-based algorithms to solve a monotone inclusion problem  in infinite dimensional Hilbert spaces. Solution theorems of strong convergence are obtained under the certain conditions. Some numerical experiments are presented to illustrate that  our algorithms are efficient than the existing results.
\end{abstract}


\begin{MSC}
47H05, 49J40, 65K10,   47J20.
\end{MSC}

\thispagestyle{first}\renewcommand{\thefootnote}{\fnsymbol{footnote}}
\footnotetext{\hspace*{-5mm}
\renewcommand{\arraystretch}{1}
\begin{tabular}{@{}r@{}p{10cm}@{}}
$^\dag$& The corresponding author. Email address: qxlxajh@163.com (Xiaolong Qin)\\
$^1$&Institute of Fundamental and Frontier Sciences,
University of Electronic Science and Technology of China, Chengdu, China.\\
$^2$&Department of Mathematics, Zhejiang Normal University, Zhejiang, China.\\
\end{tabular}}

\vspace{-2mm}

\section{Introduction}

In this paper, we start with the general optimization problem
\begin{equation}\label{eq:P}
\min \{\Phi(\mathbf{x})=F(\mathbf{x})+G(\mathbf{x}) : x \in H\}\,,      \tag{P}
\end{equation}
where $ H $ is a real Hilbert space with  inner product $\langle\cdot, \cdot\rangle$ and  induced norm $\|\cdot\|$, $F : H \rightarrow(-\infty, \infty)$ is a continuously differentiable function  and  $G : H \rightarrow(-\infty, \infty]$ is a convex and closed function, which is assumed to be subdifferentiable on $\operatorname{dom} G$, the domain of $G$.
 If $\mathbf{x}^{\star} \in H$  is a
local minimum of \eqref{eq:P}, then it is a stationary point of \eqref{eq:P}, i.e.,
\begin{equation}\label{eq1}
\mathbf{0} \in \nabla F\left(\mathbf{x}^{\star}\right)+\partial G\left(\mathbf{x}^{\star}\right)\,,
\end{equation}
where $\partial G(\cdot)$ stands for the subdifferential of $G$. Note that if $ F $ is also convex,  then $\mathbf{x}^{\star}$  is the global minimum of \eqref {eq:P}. For any $t>0$, one sees from   \eqref{eq1}  that
\[
\begin{aligned}
\mathbf{0}  &\in t \nabla F\left(\mathbf{x}^{\star}\right)+t \partial G\left(\mathbf{x}^{\star}\right) \Leftrightarrow (I+t \partial G)\left(\mathbf{x}^{\star}\right)\\
&\in(I-t \nabla F)\left(\mathbf{x}^{\star}\right) \Leftrightarrow \mathbf{x}^{\star}=(I+t \partial G)^{-1}(I-t \nabla F)\left(\mathbf{x}^{\star}\right)\,,
\end{aligned}
\]
from which a  fixed-point scheme naturally arises  to generate the following iterative  sequence $\left\{\mathbf{x}_{k}\right\}$:
 \begin{equation}\label{eq2}
 \mathbf{x}_{k}=\left(I+t_{k} \partial G\right)^{-1}\left(I-t_{k} \nabla F\right)\left(\mathbf{x}_{k-1}\right)\,, \quad  \mathbf{x}_{0}\in \mathbb{R}\,, \,t_{k}>0\,.
 \end{equation}
Actually, \eqref{eq2}   is a special case of the  forward-backward  (FB)  algorithm  which was originally designed to find a zero of the more general inclusion problem:
\begin{equation}\label{origin}
\mathbf{0} \in A\left(\mathbf{x}^{\star}\right)+B\left(\mathbf{x}^{\star}\right)\,,
\end{equation}
where $ A $ and $ B $ are set-valued maximal monotone  maps. \eqref{origin} is reduced to  \eqref{eq1} if both
$ F $ and $G $ are convex and $A :=\nabla F$ and $B :=\partial G$.

A classic algorithm to solve \eqref{origin} is the known  forward-backward splitting algorithm,   which was first introduced by Passty~\cite{passty1979ergodic}, and Lions and Mercier~\cite{lions1979splitting}. In recent years, this method has been widely investigated in various problems, such as, coupled monotone inclusions, constrained variational inequalities, signal processing, image recovery, machine learning, convex optimization problems, etc; see, e.g., \cite{attouch2010parallel,combettes2005signal,ochs2014ipiano,duchi2009efficient} and the references therein. It is known that the FB method converges provided that   the inverse of forward mapping $ A^{-1} $ is strongly monotone and $B$ is maximal monotone~\cite{gabay1983augmented}. In 1997, Chen and Rockafellar~\cite{chen1997convergence}  gave the convergence rates analysis of the FB method. In 2000, Tseng~\cite{tseng2000modified} obtained a modified FB algorithm for zeros of  maximal monotone mappings. This method achieves convergence only by assuming that the forward mapping is continuous over a closed convex subset of its domain.

\mybox{
\textbf{Tseng (2000) : Modified Forward-Backward Splitting Algorithm.}
\begin{equation}\label{algtseng}
\left\{\begin{array}{l}
{y_{n}=\left(I+\gamma_{n} G\right)^{-1}\left(x_{n}-\gamma_{n} F x_{n}\right)}\,, \\
{x_{n+1}=y_{n}-\gamma_{n}\left(F y_{n}-F x_{n}\right)}\,.\end{array}\right.
\end{equation}
}

Polyak~\cite{polyak1964some} first proposed the inertial idea  to improve the convergence of the algorithms. Inertial-type methods,  which are  considered as a method to accelerate the convergence of Tseng-type iterative methods, are based on a discrete version of a second-order dissipative dynamical system~\cite{attouch2000heavy}.  In recent years, some authors constructed various fast iterative algorithms via inertial extrapolation techniques on some classical methods, such as,  inertial proximal point algorithms, inertial Mann algorithms,  inertial Douglas-Rachford splitting  algorithms, inertial alternating direction method of multipliers, inertial forward-backward splitting algorithms  and inertial extragradient algorithms, etc. On the other hand,   Nesterov~\cite{Nesterov1983A} developed an acceleration scheme which improves the  convergence speed of the forward-backward algorithm  from the standard $O\left(k^{-1}\right)$  to $O\left(k^{-2}\right)$. In addition, Attouch and Peypouquet~\cite{attouch2016rate} proved that the Nesterov's accelerated forward-backward method is actually $o\left(k^{-2}\right)$  rather than  $O\left(k^{-2}\right)$.

In 2015, Lorenz and Pock~\cite{pock} proposed the following  inertial forward-backward algorithm by combining the inertial idea with the forward-backward algorithm for monotone operators. It should be noted  that Algorithm \eqref{alg13} is still   weakly convergent.

\mybox{
	\textbf{Lorenz and Pock (2015) : Inertial Forward-Backward Algorithm.}
	\begin{equation}\label{alg13}
	\left\{\begin{array}{l}{y_{n}=x_{n}+\alpha_{n}\left(x_{n}-x_{n-1}\right)}\,, \\ {x_{n+1}=\left(I+\gamma_{n} G\right)^{-1}\left(y_{n}-\gamma_{n} F y_{n}\right)}\,. \end{array}\right.
	\end{equation}		
}

In practical applications, many problems, such as, quantum physics and   image reconstruction, are in  infinite dimensional spaces.
To investigate these problems,   norm convergence is usually preferable to the weak convergence.
In 2003, Nakajo and Takahashi~\cite{nakajo2003strong} established strong convergence of the Mann iteration with the aid of projections.
 Indeed, they considered the following algorithm:

\mybox{
	\textbf{Nakajo and Takahashi (2003) : Hybrid Projection Method.}
\begin{equation}\label{alg141}
\left\{\begin{array}{l}{y_{n}=\alpha_{n} x_{n}+\left(1-\alpha_{n}\right) T x_{n}}\,, \\ {C_{n}=\left\{u \in C:\left\|y_{n}-u\right\| \leqslant\left\|x_{n}-u\right\|\right\}}\,, \\ {Q_{n}=\left\{u \in C:\left\langle x_{n}-u, x_{0}-x_{n}\right\rangle \geqslant 0\right\}}\,, \\ {x_{n+1}=\mathcal{P}_{C_{n} \cap Q_{n}} x_{0}\,, \quad n \in \mathbb{N}}\,,\end{array}\right.
\end{equation}
}
\noindent where $\left\{\alpha_{n}\right\} \subset[0,1)$, $ T $ is a nonexpansive mapping on $C$ and  $\mathcal{P}_{C_{n} \cap Q_{n}}$ is the nearest point projection from $C$ onto $C_{n} \cap Q_{n}$. This method is now referred as  the hybrid projection method.
Inspired by Nakajo and Takahashi~\cite{nakajo2003strong},   Takahashi, Takeuchi and Kubota~\cite{takahashi2008strong} also proposed
a projection-based method and obtain the strong convergence of the method, which is now  called the  shrinking projection method.
In recent years, many authors studied these projection-based methods in various spaces; see, e.g., \cite{plubtieng2007strong,kim,dong,ziming}.

Inspired and motivated by the above works, we propose two new projection-based inertial solution methods with   adaptive stepsizes, which are more flexible than the fixed stepsizes. Solution theorems of strong convergence are established in the framework of real Hilbert spaces. Numerical examples to illustrate the efficiency and robustness of the proposed algorithms are provided.
Our paper is organized as follows. In Section \ref{sec2}, we give some useful and necessary preliminaries for our convergence analysis and numerical experiments. In Section \ref{sec3}, we propose our new algorithms, and obtain solution theorems of strong convergence  under some mild conditions. In Section \ref{sec4}, we give some numerical results in convex minimization problems to show the efficient and robust of our algorithms.  Section \ref{sec5} ends this paper.

\section{Preliminaries}\label{sec2}
Let $ C $ be a non-empty, convex and closed set in a real Hilbert space $ H $. For a given sequence $\left\{x_{n}\right\} \subset H$, let $\omega_{w}\left(x_{n}\right):=\left\{x: \exists x_{n_{j}} \rightharpoonup x\right\}$ denote the weak $ w $-limit set of $\{x_{n}\}$. For any $ x, y \in H $,  we have
\begin{enumerate}[(1)]
	\item $\|x-y\|^{2}+2\langle x-y, y\rangle =\|x\|^{2}-\|y\|^{2}$;
	\item $\|x+y\|^{2} \leq\|x\|^{2}+2\langle y, x+y\rangle $;
	\item $\|t x+(1-t) y\|^{2}+t(1-t)\|x-y\|^{2} =t\|x\|^{2}+(1-t)\|y\|^{2}, \,\forall t\in\mathbb{R} $.
\end{enumerate}

Let $F : H \rightarrow H$ be an operator. The fixed-pint set of $F$ is denoted by $\operatorname{Fix}(F)$, where $\operatorname{Fix}(F):=\{x \in H \mid F x=x\}$.  $F$ is said to be $L$-Lipschitz continuous with $L>0$ if
\[
\|F x-F y\| \leq L\|x-y\| \,, \,\forall x, y \in H\,.
\]
If $L=1$, then $ F $ is said to be nonexpansive.
$ F $ is said to be monotone if
\[
\langle F x-F y, x-y\rangle \geq 0 \,, \,\forall x, y \in H\,.
\]
$ F $ is said to be strongly monotone with $L>0$ if
\[
\langle F x-F y, x-y\rangle \geq L\|x-y\|^2\,, \,\forall x, y \in H\,.
\]		

For  any $ x \in H$, there exists a unique nearest point in $C$, denoted by $\mathcal{P}_{C} x$, such that
\[
\mathcal{P}_{C}(x) {:=} \operatorname{argmin}_{y \in C}\|x-y\|\,,
\]
where $\mathcal{P}_{C}$ is called the metric projection of $ H $ onto $ C $. It has such an equivalent form $ \langle \mathcal{P}_{C} x-x, \mathcal{P}_{C} x-y\rangle \le 0, \forall y \in C$, and can also be converted to $\left\|y-\mathcal{P}_{C} x\right\|^{2}+\left\|x-\mathcal{P}_{C} x\right\|^{2} \leq\|x-y\|^{2}$.
It can be calculated that the projection of $x_{0}$ on a polyhedron is described by linear inequalities  $A x \preceq b$ via the following quadratic programming (QP)
\[
{\text { minimize }}  {\left\|x-x_{0}\right\|_{2}^{2}}\,, \quad  {\text { subject to }} {A x \preceq b}\,.
\]

We next give some special cases with simple analytical solutions.
\begin{enumerate}[(i)]
\item The Euclidean projection of $x_{0}$  onto an affine subspace $\Omega=\{x : A x=b\}$  with  $A \in \mathbb{R}^{m \times n}$ and  $\operatorname{rank}(A)=m<n$ is given
by
\[
\mathcal{P}_{\Omega}(x_{0})=x_{0}+A^{\mathsf{T}}\left(A A^{\mathsf{T}}\right)^{-1}(b-A x_{0})\,.
\]

\item The Euclidean projection of $x_{0}$  onto a halfspace $ \Omega=\left\{x : a^{\mathsf{T}} x \leq b \;(a \neq 0)\right\}  $ is given
by
 \[
\mathcal{P}_{\Omega}(x_{0}) =\left\{\begin{aligned}&x_{0}\,, & {\text { if }a^{\mathsf{T}} x_{0} \leq b}\,;\\&x_{0}+\frac{b-a^{\mathsf{T}} x_{0}}{\|a\|^{2}} a\,, & {\text { if } a^{\mathsf{T}} x_{0}>b}\,.\end{aligned}\right.
 \]

\end{enumerate}

Let $G$ be a proper,
lower semi-continuous and convex function.
\[
\operatorname { prox } _ { \gamma G } ( y )  { := } \operatorname { argmin } _ { x \in H }   \left\{\frac { 1 } { 2 } \| x - y \| ^ { 2 }+\gamma G ( x )\right\}\,, \,\forall  y\in H\,,
\]
where $\gamma$ is a positive real number,
is called   the proximity operator.

 Note that it has the closed-form expression in some important cases. For example,
if the Euclidean norm $ G(x)=\|x\|_1 $, then one has  the shrinkage-threshold operator $\mathcal{T}_{\gamma}(y)$
\[
\operatorname{prox}_{\gamma G}(y) = \left(\mathcal{T}_{\gamma}(y)\right)_{i}=\left\{\begin{aligned}&\operatorname{sign}\left(y_{i}\right) \cdot\left(\left|y_{i}\right|-\gamma\right)_{+}\,,&\text{if } \left|y_{i}\right|>\gamma\,; \\ &0 \,,&\text{if } \left|y_{i}\right| \leq \gamma\,.\end{aligned}\right.
\]

Let $G : H \rightarrow 2^{H}$ be a multivalued operator on $H$. $G$
is said to be  monotone  iff  $\langle p-q, x-y\rangle \geq 0$ for any $ x, y \in H$, $ p\in Gx $ and $q \in Gy$.   $G : H \rightarrow 2^{H}$
Recall that a mutivalued operator is said to be  maximal  iff its  Graph  is not
contained in the graph of any other monotone operator properly.
One knows that a monotone $G : H \rightarrow 2^{H}$  is maximal iff  for any $ (x, p) \in H \times H$, $\langle p-q, x-y\rangle \geq 0$  for every $(y, q) \in \operatorname{Graph}(G)$  yields $p \in G x$.

\begin{lemma}\cite{qinoptim}\label{lemmaw}
Let $F : H \rightarrow H$  be a operator  and $G : H \rightarrow 2^{H}$ be a maximal monotone operator. Define $T_{\gamma} :=(I+\gamma G)^{-1}(I-\gamma F),\gamma>0$. Then
, $ \operatorname{Fix}\left(T_{\gamma}\right)=(F+G)^{-1}(0) $.
\end{lemma}
\begin{lemma}\cite{brezis1973ope}\label{lemmae}
Let $F : H \rightarrow H$ be a Lipschitz continuous and monotone mapping, and let $G : H \rightarrow 2^{H}$ be a maximal monotone mapping. Then  $ F+G $  is   maximally monotone.
\end{lemma}

\begin{lemma}\cite{kim2006strong}\label{lemmaa}
Let $ C $ be a   convex and closed set in  a real Hilbert space $ H $. Given $x, y, z \in H$ and $a \in \mathbb{R}$,
$ \left\{v \in C :\|y-v\|^{2} \leq\|x-v\|^{2}+\langle z, v\rangle+ a\right\} $
is convex and closed.
\end{lemma}

\begin{lemma}\label{lemma0} \cite{martinez2006strong}
Let $ C $ be a   convex and closed set in a real Hilbert space  $ H $, $\left\{x_{n}\right\}\subset H $ and $u \in H$. Let $q=\mathcal{P}_{C} u$. If the weak $\omega$-limit set $\omega_{w}\left(x_{n}\right) \subset C$ and
$\left\|x_{n}-u\right\| \leq\|u-q\|,$ $\forall n\in \mathbb{N},$
then $\{x_{n}\}$ converges to $q$ in norm.
\end{lemma}

\section{Main Results}\label{sec3}

In this section, we   assume that the following conditions are satisfied for our convergence analysis.
\begin{description}
\item[\textbf{(A1)}]  The solution set of the inclusion problem \eqref{origin} is nonempty, i.e., $ \Omega := (F+G)^{-1}(0) \neq~\emptyset $.
\item[\textbf{(A2)}]  The mapping $G : H \rightarrow 2^{H}$ is maximal monotone, $F : H \rightarrow H$ is $ L $-Lipschitz continuous and monotone.
\end{description}

\subsection{The Inertial Hybrid Projection Algorithm}

\begin{algorithm}[H]
	\textbf{Algorithm 3.1: Inertial Hybrid Projection Algorithm (IHPA).}\\
	\label{CQ}
	\KwIn{$x_{-1} = x_{0}$,
		$\gamma_{0}>0$, $ \mu \in (0,1) $,  $ \alpha_{n}\in[0,1) $.}
	\For{$n =0:\text{Maxiters}$}
	{
		\begin{equation}\label{eq:CQ}
		\left\{\begin{aligned}
		w_{n}&=x_{n}+\alpha_{n}\left(x_{n}-x_{n-1}\right)\,,\\
		y_{n}&=\left(I+\gamma_{n} G\right)^{-1}\left(I-\gamma_{n} F\right) w_{n}\,, \\
		z_{n} &=y_{n}-\gamma_{n}\left(F y_{n}-F w_{n}\right)\,, \\
		C_{n} &=\Big\{u \in H :\left\|z_{n}-u\right\|^2 \leq\left\|w_{n}-u\right\|^2-\big( 1-\mu^{2}\frac{\gamma_{n}^{2}}{\gamma_{n+1}^{2}}\big)\left\|w_{n}-y_{n}\right\|^{2}\Big\},\\
		Q_{n} &=\left\{u \in H :\left\langle x_{n}-u, x_{n}-x_{0}\right\rangle \leq 0\right\}\,,\\
		x_{n+1}&=\mathcal{P}_{C_{n} \cap Q_{n}} x_{0} \,, \, n\ge 0\,.\end{aligned}\right.
		\end{equation}
		Update $ \gamma_{n} $ by \eqref{update},
	}
\end{algorithm}
\noindent where $\{\gamma_{n}\}$ is the stepsize generated by
\begin{equation}\label{update}
\gamma_{n+1}=\left\{\begin{aligned}&\min \left\{\frac{\mu\left\|w_{n}-y_{n}\right\|}{\left\|F w_{n}-F y_{n}\right\|}, \gamma_{n}\right\}\,, &{\text{if } F w_{n}-F y_{n} \neq 0}\,; \\ &\gamma_{n}\,, &{\text{otherwise}}\,.\end{aligned}\right.
\end{equation}

\begin{remark}
The inertial parameter $ \{\alpha_{n}\} $ in \eqref{eq:CQ} can be selected as an arbitrary sequence in $ [0,1) $ to produce acceleration. Notice that the parameter  $\{\alpha_{n}\}$ in \eqref{eq:CQ} was generated by the expression $\big(\frac{t_{n-1}-1}{t_{n}}\big)$ in~\cite{beck2009fast} and $\frac{n-1}{n+3}$ in~\cite{attouch2016rate}. In this paper, $ \{\alpha_{n}\}$ will also be adaptively updated by
\begin{equation}\label{alpha}
{\alpha}_{n}=\left\{\begin{aligned} &\min \left\{\alpha,  \frac{\xi_{n}}{\left\|x_{n}-x_{n-1}\right\|}\right\}\,, &{\text{if } x_{n} \neq x_{n-1}}\,; \\ &\alpha\,, &{\text{otherwise}}\,,\end{aligned}\right.
\end{equation}
where $ \alpha\in[0,1) $,  the sequence $ \{\xi_n\} $ satisfies $ \lim_{n \rightarrow \infty}  \xi_n=0$ and $ \sum_{n=1}^{\infty} \xi_{n}=\infty $.
\end{remark}

The following lemmas play a significant role in this paper for the convergence analysis.

\begin{lemma}\label{lemma2}
Let $\left\{z_{n}\right\}$ be a sequence generated
	by Algorithm 3.1. If 	  conditions (A1) and (A2) hold, then
	\begin{equation}\label{q1}
	\left\|z_{n}-p\right\|^{2} \leq\left\|w_{n}-p\right\|^{2}-\bigg(1-\mu^{2} \frac{\gamma_{n}^{2}}{\gamma_{n+1}^{2}}\bigg)\left\|w_{n}-y_{n}\right\|^{2} \,, \, \forall p \in \Omega\,.
	\end{equation}
\end{lemma}

\begin{proof}
Setting $ a_{n} = \gamma_{n}^{2}\left\|F y_{n}-F w_{n}\right\|^{2}  -2 \gamma_{n}\left\langle y_{n}-p, F y_{n}-F w_{n}\right\rangle $, one has
\begin{equation}\label{w}
\begin{aligned}\left\|z_{n}-p\right\|^{2} =&\left\|y_{n}-p\right\|^{2}+\gamma_{n}^{2}\left\|F y_{n}-F w_{n}\right\|^{2}-2 \gamma_{n}\left\langle y_{n}-p, F y_{n}-F w_{n}\right\rangle\\ =&\left\|w_{n}-p\right\|^{2}+\left\|y_{n}-w_{n}\right\|^{2}+2\left\langle w_{n}-p,y_{n}-w_{n}\right\rangle + a_{n}\\ =&\left\|w_{n}-p\right\|^{2}+\left\|y_{n}-w_{n}\right\|^{2}-2\left\langle y_{n}-w_{n}, y_{n}-w_{n}\right\rangle + 2\left\langle y_{n}-w_{n}, y_{n}-p\right\rangle + a_{n}\\
=&\left\|w_{n}-p\right\|^{2}-\left\|y_{n}-w_{n}\right\|^{2}-2\left\langle y_{n}-p, w_{n}-y_{n}+\gamma_{n}\left(F y_{n}-F w_{n}\right)\right\rangle\\ 
&+\gamma_{n}^{2}\left\|F y_{n}-F w_{n}\right\|^{2}\,.
\end{aligned}
\end{equation}	
Note that
\[
	\gamma_{n+1}=\min \left\{\frac{\mu\left\|w_{n}-y_{n}\right\|}{\left\|F w_{n}-F y_{n}\right\|}, \gamma_{n}\right\} \leq \frac{\mu\left\|w_{n}-y_{n}\right\|}{\left\|F w_{n}-F y_{n}\right\|}\,,
\]
which means that
\begin{equation}\label{q}
	\left\|F w_{n}-F y_{n}\right\| \leq \frac{\mu}{\gamma_{n+1}}\left\|w_{n}-y_{n}\right\|\,.
\end{equation}
If $F w_{n}=F y_{n}$, then   inequality
\eqref{q}  holds obviously.
Combining  \eqref{w} and \eqref{q}, one obtains
\begin{equation}\label{e}
\begin{aligned}\left\|z_{n}-p\right\|^{2} \leq \left\|w_{n}-p\right\|^{2}-\bigg(1-\mu^{2} \frac{\gamma_{n}^{2}}{\gamma_{n+1}^{2}}\bigg)\left\|w_{n}-y_{n}\right\|^{2} -2\left\langle y_{n}-p,w_{n}-y_{n}+\gamma_{n}\left(F y_{n}-F w_{n}\right)\right\rangle\,. \end{aligned}
\end{equation}
	
Next, one proves
\begin{equation}\label{r}
\left\langle y_{n}-p,w_{n}-y_{n}+\gamma_{n}\left(F y_{n}-F w_{n}\right)\right\rangle \geq 0\,.
\end{equation}
From $ y_{n}=\left(I+\gamma_{n} G\right)^{-1}\left(I-\gamma_{n} F\right) w_{n} $, one obtains $\left(I-\gamma_{n} F\right) w_{n}\in\left(I+\gamma_{n} G\right) y_{n} $. Since 	$G$ is maximally monotone, one concludes that there exists $u_{n} \in G y_{n}$ such that
$\left(I-\gamma_{n} F\right) w_{n}=y_{n}+\gamma_{n} u_{n}$. This means that
\begin{equation}\label{t}
u_{n}=\frac{1}{\gamma_{n}}\left(w_{n}-\gamma_{n} F w_{n}-y_{n}\right)\,.
\end{equation}	
On the other hand, one has $ 0 \in(F+G) p $ and $ F y_{n}+u_{n} \in(F+G) y_{n} $. Since $ F+G $  is maximally monotone, one gets
\begin{equation}\label{y}
\left\langle F y_{n}+u_{n}, y_{n}-p\right\rangle \geq 0\,.
\end{equation}
Substituting \eqref{t} into \eqref{y}, one gets
\[
\frac{1}{\gamma_{n}}\left\langle w_{n}-\gamma_{n} F w_{n}-y_{n}+\gamma_{n} F y_{n}, y_{n}-p\right\rangle \geq 0\,,
\]
which means that
$\left\langle w_{n}-y_{n}+\gamma_{n}\left(F y_{n}-F w_{n}\right), y_{n}-p\right\rangle \geq 0 $.
From \eqref{e} and \eqref{r}, one concludes    \eqref{q1} immediately.
\end{proof}

\begin{lemma}\label{lemma3}
Let $\left\{x_{n}\right\},\left\{w_{n}\right\}$ and $\left\{y_{n}\right\}$  be   three sequences generated by Algorithm 3.1.
Assume that conditions (A1) and  (A2) hold.
If  $\lim _{n \rightarrow \infty}\left\|w_{n}-x_{n}\right\|=\lim _{n \rightarrow \infty}\left\|w_{n}-y_{n}\right\|=~0$, and $\left\{x_{n_{k}}\right\}$, which is a subsequence of $\{x_n\}$,  converges weakly to
	some $q \in H$, then $q \in \Omega$, where $ \Omega = (F+G)^{-1}(0)$.
\end{lemma}

\begin{proof}
Let $(h, g) \in \operatorname{Graph}(F+G)$, i.e., $g-F h \in G h$. Since
$y_{n_{k}}= \left(I+\gamma_{n_{k}} G\right)^{-1}\left(I-\gamma_{n_{k}} F\right) w_{n_{k}} $, one obtains
$\left(I-\gamma_{n_{k}} F\right) w_{n_{k}} \in\left(I+\gamma_{n_{k}} G\right) y_{n_{k}} $, which implies
\[
\frac{1}{\gamma_{n_{k}}}\left(w_{n_{k}}-y_{n_{k}}-\gamma_{n_{k}} F w_{n_{k}}\right) \in G y_{n_{k}}\,.
\] 
On the other hand, by the maximal monotonicity of $ G $, one has
\[
\left\langle h-y_{n_{k}}, g-F h-\left(w_{n_{k}}-y_{n_{k}}-\gamma_{n_{k}} F w_{n_{k}}\right)/{\gamma_{n_{k}}}\right\rangle \geq 0\,.
\]
Therefore,
\[
\begin{aligned}\left\langle h-y_{n_{k}}, g\right\rangle & \geq\left\langle h-y_{n_{k}}, F h+\left(w_{n_{k}}-y_{n_{k}}-\gamma_{n_{k}} F w_{n_{k}}\right)/{\gamma_{n_{k}}}\right\rangle \\ &=\left\langle h-y_{n_{k}}, F h-F w_{n_{k}}\right\rangle+\left\langle h-y_{n_{k}}, \left(w_{n_{k}}-y_{n_{k}}\right)/{\gamma_{n_{k}}}\right\rangle \\ &=\left\langle h-y_{n_{k}}, F h-F y_{n_{k}}\right\rangle+\left\langle h-y_{n_{k}}, F y_{n_{k}}-F w_{n_{k}}\right\rangle +\left\langle h-y_{n_{k}}, \left(w_{n_{k}}-y_{n_{k}}\right)/{\gamma_{n_{k}}}\right\rangle \\ & \geq\left\langle h-y_{n_{k}}, F y_{n_{k}}-F w_{n_{k}}\right\rangle+\left\langle h-y_{n_{k}}, \left(w_{n_{k}}-y_{n_{k}}\right)/{\gamma_{n_{k}}}\right\rangle\,. \end{aligned}
\]
Since $\lim _{n \rightarrow \infty}\left\|w_{n}-x_{n}\right\|=0$, $\lim _{n \rightarrow \infty}\left\|w_{n}-y_{n}\right\|=0$, and $ F $ is Lipschitz continuous, one gets	$\lim _{k \rightarrow \infty} \| F y_{n_{k}} -F w_{n_{k}} \|=0 $. By  $ \lim _{n \rightarrow \infty} \gamma_{n}=\gamma \geq \min \left\{\gamma_{0}, \frac{\mu}{L}\right\} $, one obtains
\[
\lim _{k \rightarrow \infty}\left\langle h-y_{n_{k}}, g\right\rangle = \langle h-q, g\rangle  \geq 0\,.
\]  
With the aid of the maximal monotonicity of $F+G$, one obtains $0\in(F+G) q$, that is, $q \in \Omega$.
\end{proof}

\begin{theorem}\label{thm:CQ}
Assume that both $ F $ and $ G $ satisfy conditions
(A1)--(A2). Then the sequence $
\left\{x_{n}\right\}
$ generated by Algorithm 3.1 converges
to an element $ q^{*} \in \Omega $ strongly, where $ q^{*} = \mathcal{P}_{\Omega}x_0 $.
\end{theorem}

\begin{proof}
The proof is divided into three steps.
	
\noindent {\it Step 1}. It is obvious that $C_{n}$ and $Q_{n}$ are  convex closed for all $n \geq 0$.
Next one  shows that $\Omega \subset C_{n} \cap Q_{n},\forall n\ge 0$ and $\left\{x_{n}\right\}$ is well defined.
Lemma~\ref{lemma2} implies that $\Omega \subset C_{n} , \forall n \geq 0$. From the definition of $ Q_{n} $ in Algorithm 3.1, one has $Q_{0}=H$. Further, $\Omega \subset C_{0} \cap Q_{0}$ and $x_{1}=\mathcal{P}_{C_{0} \cap Q_{0}} x_{0}$ is well defined. Without  loss of generality, one assumes that $x_{n}$ is given and $\Omega \subset C_{n} \cap Q_{n}$ for some $n$. This shows that $x_{n+1}=\mathcal{P}_{C_{n} \cap Q_{n}} x_{0}$ is well defined. It follows from the projection  that
$ \left\langle z-x_{n+1}, x_{0}-x_{n+1}\right\rangle \leq 0 , \forall z \in C_{n} \cap Q_{n} $.
Since  $\Omega \subset C_{n} \cap Q_{n}$, one concludes
$ \left\langle u-x_{n+1}, x_{0}-x_{n+1}\right\rangle \leq 0 , \forall u \in \Omega $.  This implies that $\Omega \subset Q_{n+1}$, and  thus $\Omega \subset C_{n+1} \cap Q_{n+1}$.

\noindent {\it Step 2}. One shows  that  $ \left\{x_{n}\right\}$ is bounded and $\lim _{n \rightarrow \infty}\left\|w_{n}-y_{n}\right\|=0$.
Since $ \Omega \subset C_{n} \cap Q_{n}  $ and  $x_{n+1}= \mathcal{P}_{C_{n} \cap Q_{n}} x_{0}$, one gets
$  \left\|x_{n+1}-x_{0}\right\| \leq\left\|q^{*}-x_{0}\right\| , \forall n \geq 0 $. This means that $ \left\{x_{n}\right\}$ is bounded,  so are $\left\{w_{n}\right\}$ and $\left\{z_{n}\right\}$.
Combining the definition of $ Q_{n} $ and the projection, one has $x_{n}=\mathcal{P}_{Q_{n}}x_{0}$.  Since $x_{n+1} \in Q_{n}$, one further has
\[
\left\|x_{n}-x_{0}\right\| \leq\left\|x_{n+1}-x_{0}\right\| \,, \,\forall n \geq 0\,.
\]
Thus $\lim _{n \rightarrow \infty}\left\|x_{n}-x_{0}\right\|$ exists. It follows that
\[
\begin{aligned}\left\|x_{n}-x_{n+1}\right\|^{2}   \leq\left\|x_{n+1}-x_{0}\right\|^{2}-\left\|x_{n}-x_{0}\right\|^{2} \,.\end{aligned}
\]
We see that
$ \lim _{n \rightarrow \infty} \left\|x_{n}-x_{n+1}\right\| = 0 $. Since
$  \left\|x_{n+1}-z_{n}\right\|  \leq\left\|w_{n}-x_{n+1}\right\| $ and $  \left\|w_{n}-x_{n}\right\| \leq\left|\alpha_{n}\right|\left\|x_{n}-x_{n-1}\right\| $, one arrives at  $ \lim _{n \rightarrow \infty} \left\|z_{n}-w_{n}\right\| \leq \lim _{n \rightarrow \infty} \left\{\left\|z_{n}-x_{n}\right\|+\left\|x_{n}-w_{n}\right\| \right\} =~0 $. Then
\[
\bigg(1-\mu^{2} \frac{\gamma_{n}^{2}}{\gamma_{n+1}^{2}}\bigg)\left\|w_{n}-y_{n}\right\|^{2} \leq\left\|w_{n}-p\right\|^{2}-\left\|z_{n}-p\right\|^{2}  \leq\left(\left\|w_{n}-p\right\|+\left\|z_{n}-p\right\|\right)\left\|z_{n}-w_{n}\right\|\,.
\]
It is clear to see that $\lim _{n \rightarrow \infty}\left\|w_{n}-y_{n}\right\|=0$.

\noindent {\it Step 3}. One shows that $\{x_{n}\}$ converges to   $ q^{*} \in \Omega$ strongly, where $q^{*}=\mathcal{P}_{\Omega} x_{0}$.
Note that
\begin{enumerate}[(1)]
	\item If $ q^{*}=\mathcal{P}_{\Omega} x_{0} $, then
	$
	\|x_{n+1} - x_{0}\| \le \| x_{0} - q^{*}\| , \,\forall n\in \mathbb{N}.
	$	
	\item
	Every sequential weak cluster point of the sequence  $\left\{x_{n}\right\}$ is in $ \Omega $, i.e., $ \omega_{w}\left(x_{n}\right) \subset \Omega$.
\end{enumerate}
By Lemma~\ref{lemma0}, one concludes that  $ \left\{x_{n}\right\} $ converges to the point $q^{*} \in \Omega$ strongly, where $ q^{*}=\mathcal{P}_{\Omega} x_{0} $. The proof is completed.
\end{proof}

\subsection{The Inertial Shrinking Projection Algorithm}

\begin{algorithm}[H]
	\textbf{Algorithm 3.2: Inertial Shrinking Projection Algorithm (ISPA).}\\
	\label{SP}
	\KwIn{$x_{-1} = x_{0}$,
		$\gamma_{0}>0$, $ \mu \in (0,1) $, $ C_{0}=H $, $ \alpha_{n}\in[0,1) $.}
	\For{$n =0:\text{Maxiters}$}
	{
		\begin{equation}\label{eq:alg}
		\left\{\begin{aligned}
		w_{n}&=x_{n}+\alpha_{n}\left(x_{n}-x_{n-1}\right)\,,\\
		y_{n}&=\left(I+\gamma_{n} G\right)^{-1}\left(I-\gamma_{n} F\right) w_{n}\,, \\
		z_{n} &=y_{n}-\gamma_{n}\left(F y_{n}-F w_{n}\right)\,, \\
		C_{n+1} &=\Big\{u \in C_{n} :\left\|z_{n}-u\right\|^{2} \leq\left\|w_{n}-u\right\|^{2}-\big(1-\mu^{2}\frac{\gamma_{n}^{2}}{\gamma_{n+1}^{2}}\big)\left\|w_{n}-y_{n}\right\|^{2}\Big\}, \\ x_{n+1} &=\mathcal{P}_{C_{n+1}} x_{0} \,, \, n\ge 0\,.\end{aligned}\right.
		\end{equation}
		Update $ \gamma_{n} $ by \eqref{update}.
	}
\end{algorithm}

\begin{theorem}
	Assume that both $ F $ and $ G $ satisfy conditions
	(A1)--(A2). Then the sequence $
	\left\{x_{n}\right\}
	$ generated by Algorithm 3.2 converges
	to an element $ q^{*} \in \Omega $ strongly, where $ q^{*} = \mathcal{P}_{\Omega}x_0 $.
\end{theorem}

\begin{proof}
From Lemma~\ref{lemma2}, one easily concludes that 	
\begin{equation*}
\left\|z_{n}-p\right\|^{2} \leq\left\|w_{n}-p\right\|^{2}-\bigg(1-\mu^{2} \frac{\gamma_{n}^{2}}{\gamma_{n+1}^{2}}\bigg)\left\|w_{n}-y_{n}\right\|^{2} \,, \,\forall p \in \Omega\,.
\end{equation*}
Since $x_{n}=\mathcal{P}_{C_{n}} x_{0}$ and $x_{n+1} = \mathcal{P}_{C_{n+1}} x_{0}  \in C_{n+1} \subset C_{n}$, we obtain
$ \left\|x_{n}-x_{0}\right\| \leq\left\|x_{n+1}-x_{0}\right\|$. On the other hand, from $\Omega \subset C_{n}$, we get
$ \left\|x_{n}-x_{0}\right\| \leq\left\|u-x_{0}\right\| $. It implies that the  $\left\{x_{n}\right\}$ is bounded and nondecreasing. Thus, $\lim _{n \rightarrow \infty}\left\|x_{n}-x_{0}\right\|$ exists.
From Step 3 in Theorem~\ref{thm:CQ},  $\lim _{n \rightarrow \infty}\left\|x_{n+1}-x_{n}\right\| =0$ and $\lim _{n \rightarrow \infty}\left\|w_{n}-y_{n}\right\| =0$ hold. From Lemma~\ref{lemma3} and Lemma~\ref{lemma0},     $ \left\{x_{n}\right\} $ converges  to the point $q^{*} \in \Omega$ strongly, where $ q^{*}=\mathcal{P}_{\Omega} x_{0} $. 
\end{proof}

\section{Numerical Results}\label{sec4}

In this section, we   give some numerical examples to illustrate the effectiveness and robustness of the proposed algorithms in Section \ref{sec3}. We compare  the two strong convergence algorithms, proposed by Gibali and Thong~\cite{gibali2018tseng},   Mann Tseng-type   algorithm  and Viscosity Tseng-type   algorithm. All the programs are performed in MATLAB2018a on a PC Desktop Intel(R) Core(TM) i5-8250U CPU @ 1.60GHz 1.800 GHz, RAM 8.00 GB.

Based on Mann and Viscosity ideas, Gibali and Thong~\cite{gibali2018tseng} presented  two modifications of the forward-backward splitting method  in  real Hilbert spaces  as follows:

\mybox{
	\textbf{Algorithm 4.1: Mann Tseng-type modification   (MTTM).}
	\label{AGENT}
	\begin{equation*}\label{eqAGENT}
	\left\{\begin{aligned}
	y_{n}&=\left(I+\gamma_{n} G\right)^{-1}\left(I-\gamma_{n} F\right) x_{n}\,, \\
	z_{n} &=y_{n}-\gamma_{n}\left(F y_{n}-F x_{n}\right)\,, \\
	x_{n+1} &=\left(1-\delta_{n}-\theta_{n}\right) x_{n}+\theta_{n} z_{n}\,.
	\end{aligned}\right.
	\end{equation*}
	Update $ \gamma_{n} $ by \eqref{update},
}
and

\mybox{
	\textbf{Algorithm 4.2: Viscosity Tseng-type modification   (VTTM).}
	\label{COGENT}
	\begin{equation*}\label{eqCOGENT}
	\left\{\begin{aligned}
	y_{n}&=\left(I+\gamma_{n} G\right)^{-1}\left(I-\gamma_{n} F\right) x_{n}\,, \\
	z_{n} &=y_{n}-\gamma_{n}\left(F y_{n}-F x_{n}\right)\,, \\
	x_{n+1}&=\delta_{n} f\left(x_{n}\right)+\left(1-\delta_{n}\right) z_{n}\,.
	\end{aligned}\right.
	\end{equation*}
	Update $ \gamma_{n} $ by \eqref{update},
}
\noindent
where $\{\delta_{n}\}$ and $\{\theta_{n}\}$ are two real sequences in $ (0,1) $ such that $ \{\theta_{n}\}\subset (a,b)\subset (0,1-\delta_{n}) $ for some $ a>0,b>0 $, $\lim _{n \rightarrow \infty} \delta_{n}=0, \sum_{n=1}^{\infty} \delta_{n}=\infty$, and  $f: H \rightarrow H$  is a contraction.

\begin{example}\label{ex0}
Let $x=\left(x_{1}, x_{2}, \ldots,x_{10}\right) \in \mathbb{R}^{10}$ and define $F: \mathbb{R}^{10} \rightarrow \mathbb{R}^{10}$ and $G: \mathbb{R}^{10} \rightarrow \mathbb{R}^{10}$ by $F x=2 x+(1,1,\ldots,1)$ and  $G x=5 x$, respectively.
It is  clear to see that $G$ is maximally monotone, and $F$ is $2$-Lipschitz continuous and monotone. After simple calculations, we obtain
\[
\left(Id+\gamma_{n} G\right)^{-1}\left(x_{n}-\gamma_{n} F x_{n}\right) = \frac{1-2 \gamma_{n}}{1+5 \gamma_{n}} x_{n}-\frac{\gamma_{n}}{1+5 \gamma_{n}}(1,1,\ldots,1)\,.
\]

Our parameters are seted as follows. The stepsizes of the four algorithms are updated by  \eqref{update} with $ \gamma_0=0.4 $ and $ \mu=0.5 $. Algorithm 3.1 updates the inertial parameters by $ \alpha_{n}=\frac{n-1}{n+3} $.  Algorithm 4.1, Algorithm 4.2  and Algorithm 3.2 updates the inertial parameters by \eqref{alpha} with $ \alpha=0.6 $ and $ \xi_{n}=\frac{1}{(n+1)^2} $. In Algorithm 4.1 and Algorithm 4.2, we set $ \delta_{n} = \frac{1}{n+1} $, $ \theta_{n} = \frac{n}{2(n+1)} $, $ f(x)=0.5x $. the maximum iteration of $ 100 $  as the stopping criterion. Fig.~\ref{fig1} shows the convergence behavior of $ \{\|x_{n}-x^*\|\} $, where $ x^* = -(1,1,\ldots,1)/7 $. The numerical results illustrate  that the inertial parameters plays a positive role in the convergence speed and the precision of the algorithms.
\end{example}
\begin{figure}[h]
	\centering
	\includegraphics[scale=0.7]{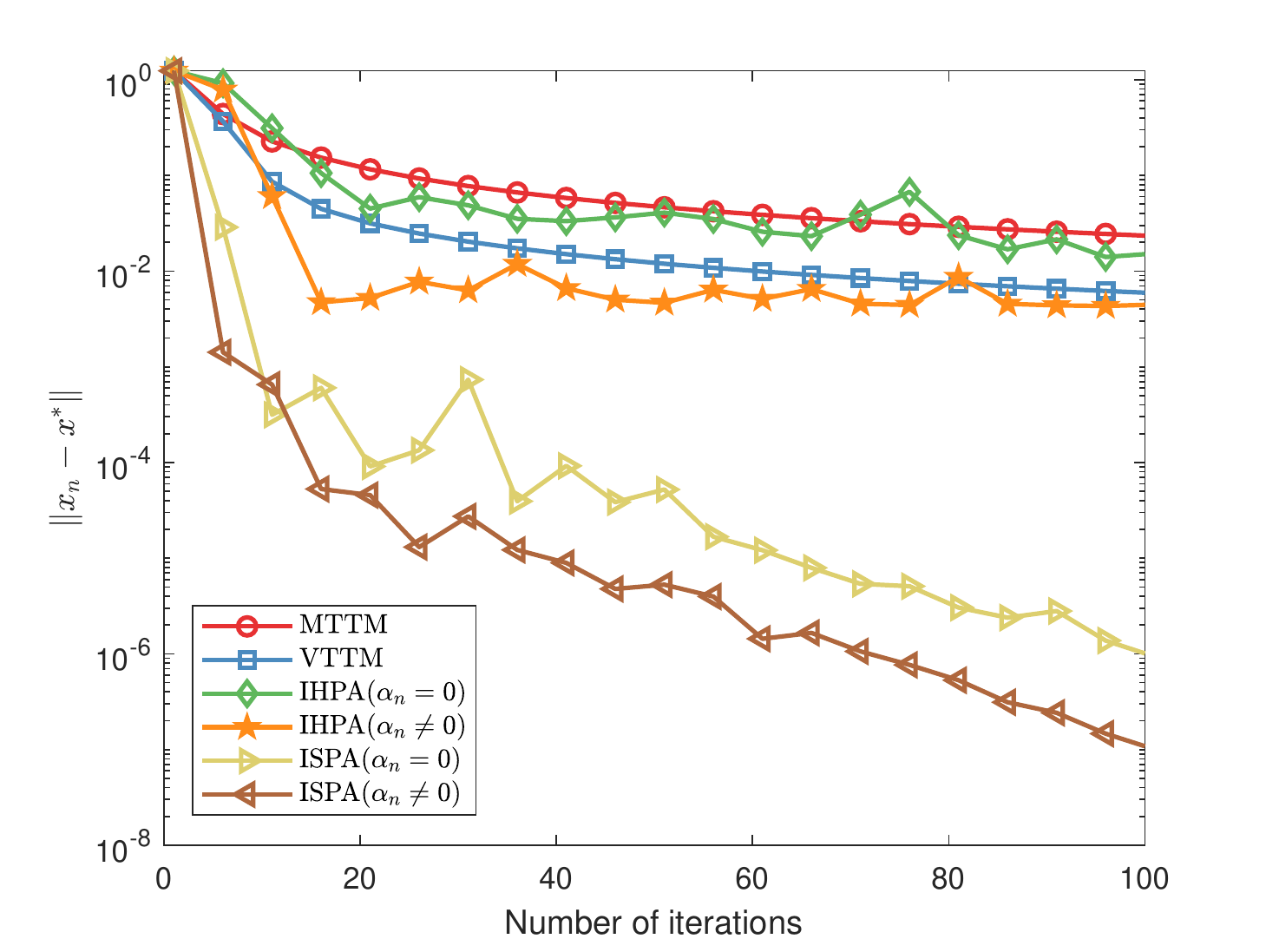}
	\caption{Convergence behavior of iterative sequences $ \{\|x_{n}-x^*\|\} $.}
	\label{fig1}
\end{figure}

\begin{example}\label{ex1}
 Find a solution of the following convex minimization problem:
 \[
 \min _{x \in \mathbb{R}^{2}} \|x\|_{2}^{2}+(3,5) x+\|x\|_{1}\,,
 \]
 where $x=\left(x_{1}, x_{2}\right) \in \mathbb{R}^{2}$. We know  the exact solution $ x^* $ is $ (-1,-2) $ and the minimum vaule is $ -5 $.

Next, we  use our algorithms to solve the minimization problem in Example~\ref{ex1}. Set $F(x)=\|x\|_{2}^{2}+(3,5)x$, $ G(x)=\|x\|_1 $ and $ \Phi(x)= F(x)+G(x)$. It is clear  that $ F $ is  convex   differentiable with $ \nabla F=2x+(3,5) $, $ G $ is convex   lower semicontinuous but not differentiable. Note that
\[
\begin{aligned}({I}+\gamma \partial G)^{-1}(x)=\left(\max \left\{\left|x_{1}\right|-\gamma, 0\right\} \operatorname{sign}\left(x_{1}\right), \max \left\{\left|x_{2}\right|-\gamma, 0\right\} \operatorname{sign}\left(x_{2}\right)\right)\,. \end{aligned}
\]
Our parameters are seted as same as in Example~\ref{ex0}.
Fig.~\ref{figxkk0}  shows the convergence behavior of the iterative sequence  $ \{\|x_{n}-x_{0}\|\} $. Fig.~\ref{figphik} shows the convergence behavior of the sequence $ \{\|\Phi(x_{n})-\Phi(x^{*})\|\} $.
\end{example}
\begin{figure}[H]
	\centering
	\includegraphics[scale=0.7]{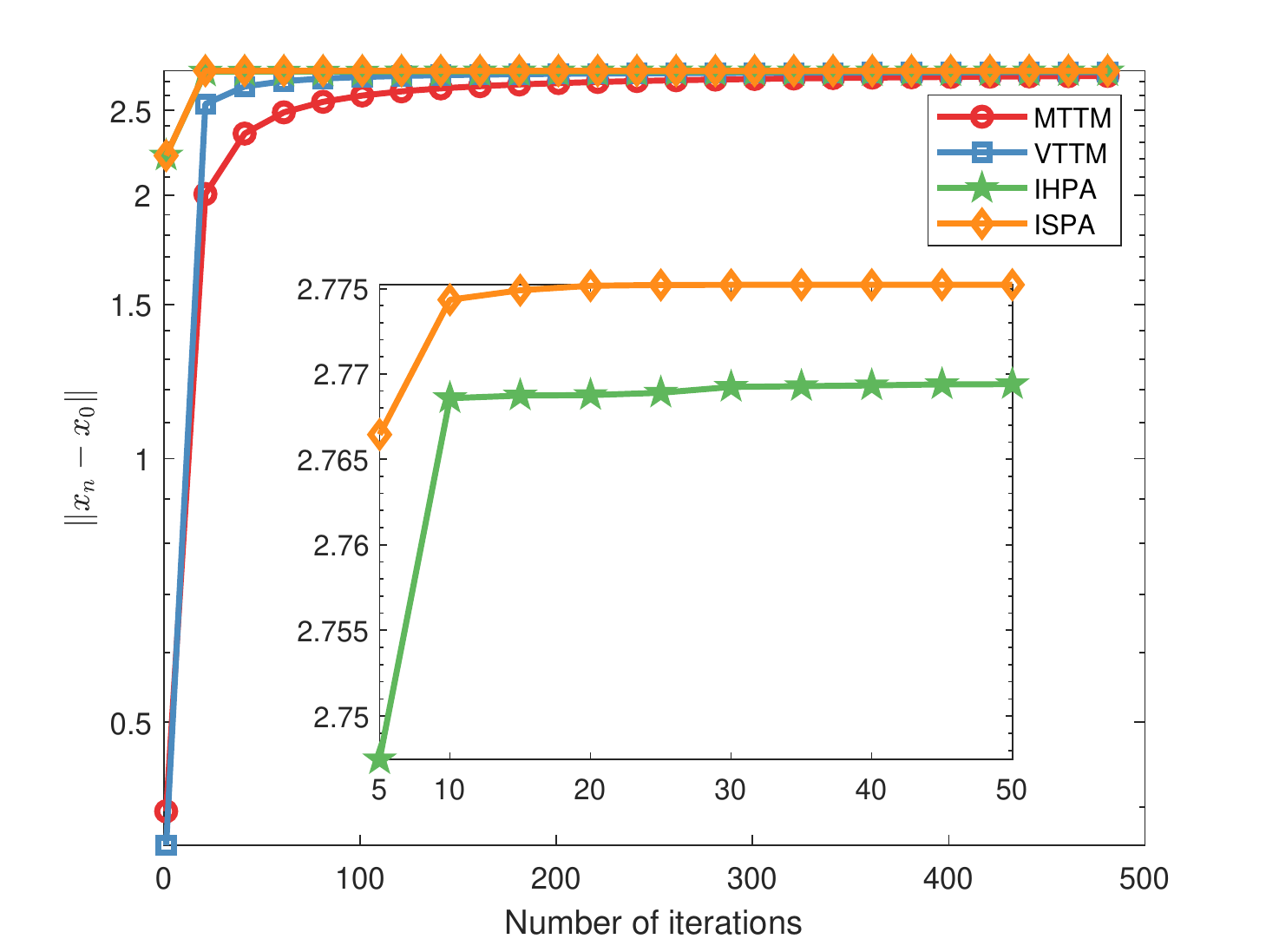}
	\caption{Convergence behavior of iterative sequences $ \{\|x_{n}-x_{0}\|\} $.}
	\label{figxkk0}
\end{figure}
\begin{figure}[H]
	\centering
	\includegraphics[scale=0.7]{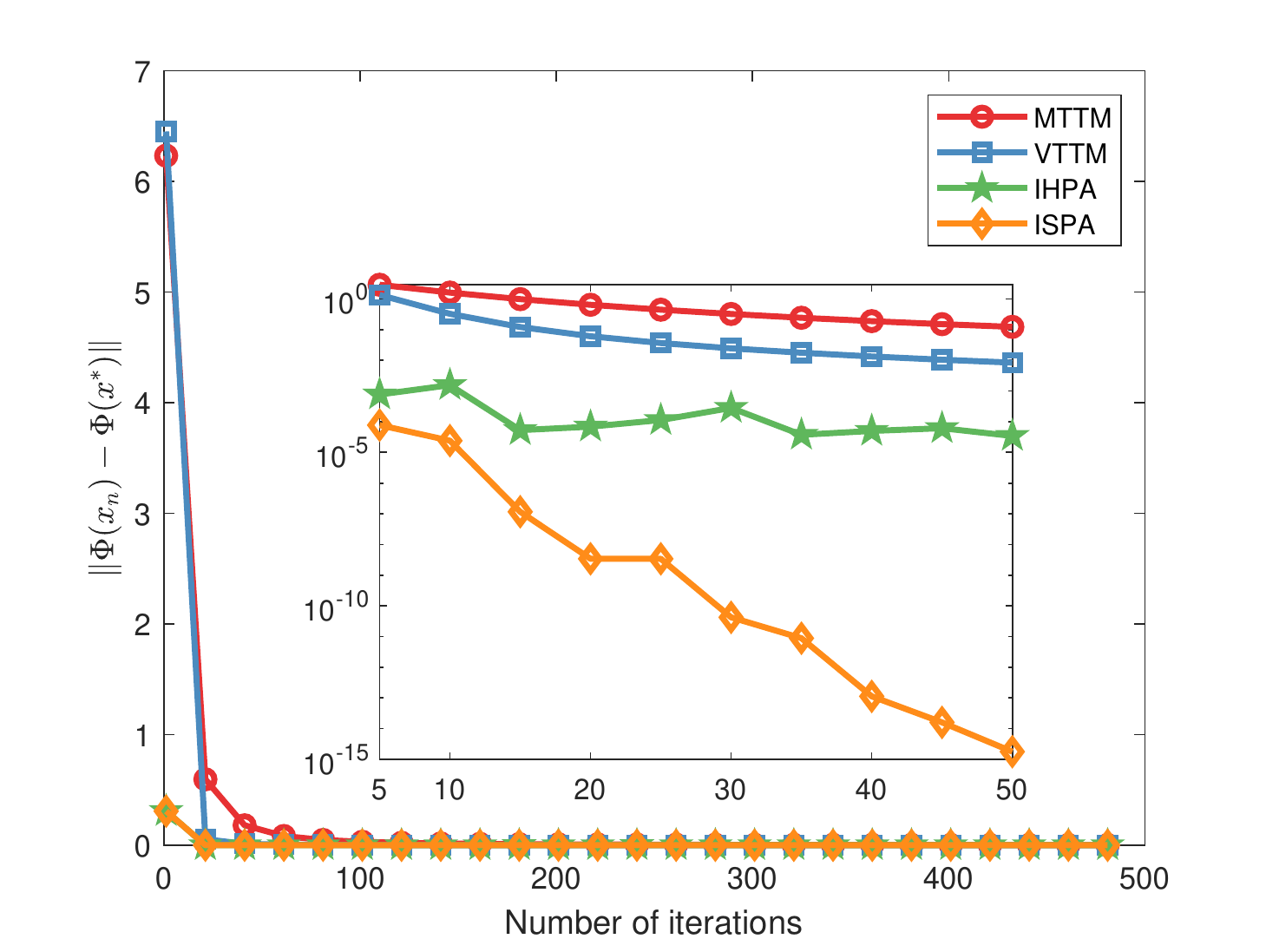}
	\caption{Convergence behavior of iterative sequences $\{\|\Phi(x_{n})-\Phi(x^{*})\|\} $.}
	\label{figphik}
\end{figure}

As shown in Figs.~\ref{figxkk0} and \ref{figphik},   sequence $\{\|\Phi(x_{n})-\Phi(x^{*})\| \}$ converges to $ 0 $ means that the function value converges to the optimal value. In addition, it is clear that the convergence speed of the iterative sequence $ \{x_{n}\} $ and $ \Phi(x_{n}) $ of  Algorithm 3.1 and Algorithm 3.2 is faster than  Algorithm 4.1 and Algorithm 4.2.

Further, we  show the numerical results  in Table~\ref{tab1}.  The function value  $ \{\Phi(x_{n})\} $ converges to the optimal value $ \Phi(x^*)=-5 $ as the number of iterations increases. We find that our proposed Algorithm 3.1 and Algorithm 3.2 enjoy higher precision than Algorithm 4.1 and Algorithm 4.2. It should be pointed out that our Algorithm 3.1 and Algorithm 3.2 require only a few iterations to achieve  convergence (cf. Table~\ref{tab1}).

\begin{table}[h!]
	\centering
	\renewcommand\arraystretch{1.5}
	\caption{Comparison of four algorithms in Example~\ref{ex1}.}
			\begin{tabular}{ccccc}
				\toprule
				\multirow{2}[4]{*}{iter $n$} &  \multicolumn{4}{c}{$\|\Phi(x_n)-\Phi(x^*)\|$} \\
				\cmidrule{2-5}          & MTTM & VTTM & IHPA & ISPA  \\
				\midrule
				1     &     $6.2330e+00$ & $6.4518e+00$ & $3.0808e-01$ & $3.0808e-01$ \\
				10    &    $1.6077e+00$ & $3.2709e-01$ & $1.5545e-03$ & $2.4074e-05$ \\
				20    &    $6.4395e-01$ & $6.0776e-02$ & $6.8905e-05$ & $3.4213e-09$ \\
				100   &    $3.1369e-02$ & $2.0442e-03$ & $4.1663e-05$ & $2.5848e-10$ \\
				300   &    $3.5320e-03$ & $2.2375e-04$ & $2.2536e-05$ & $1.7764e-15$ \\
				500   &    $1.2749e-03$ & $8.0326e-05$ & $1.6109e-05$ & $8.8818e-16$ \\
				\bottomrule
			\end{tabular}
	\label{tab1}	
\end{table}

To show that our algorithms are robust, four different initial values were tested, and the experimental results are reported in Table~\ref{tab2}.
\begin{table}[h!]
	\centering
	\renewcommand\arraystretch{1.5}
	\caption{Function value errors at different initials.}
	\begin{tabular}{ccccc}
		\toprule
		\multirow{2}[4]{*}{Start point $x_{0}$} & \multicolumn{4}{c}{$\|\Phi(x_n)-\Phi(x^*)\|$} \\
		\cmidrule{2-5}        & MTTM & VTTM & IHPA & ISPA \\
		\midrule
		$[0.6787 ,0.7577]$ & $1.2749e-03$ & $8.0326e-05$ & $2.1152e-05$ & $8.8818e-16$ \\
		$[-0.6739,-0.2305]$ & $1.2749e-03$ & $8.0326e-05$ & $2.7860e-05
		$ & $8.8818e-16$ \\
		$[0.4218,-0.9157]$ & $1.2749e-03$ & $8.0326e-05$ & $8.4837e-06$ & $1.7764e-15$ \\
		$[-0.9575,0.9649]$ & $1.2749e-03$ & $8.0326e-05$ & $1.4506e-05$ & $1.7764e-15$ \\
		\bottomrule
	\end{tabular}%
	\label{tab2}%
\end{table}%

In addition, we also plot the convergence behavior of $ \{\|\Phi(x_{n})-\Phi(x^{*})\|\} $ with different initial points in Fig.~\ref{figcom4}. Note that the projection type algorithms converge faster than the others. These results are independent of the choice of  initial values. This shows that our algorithms are effective and robust.
\begin{figure}[h!]
	\centering
	\includegraphics[scale=0.45]{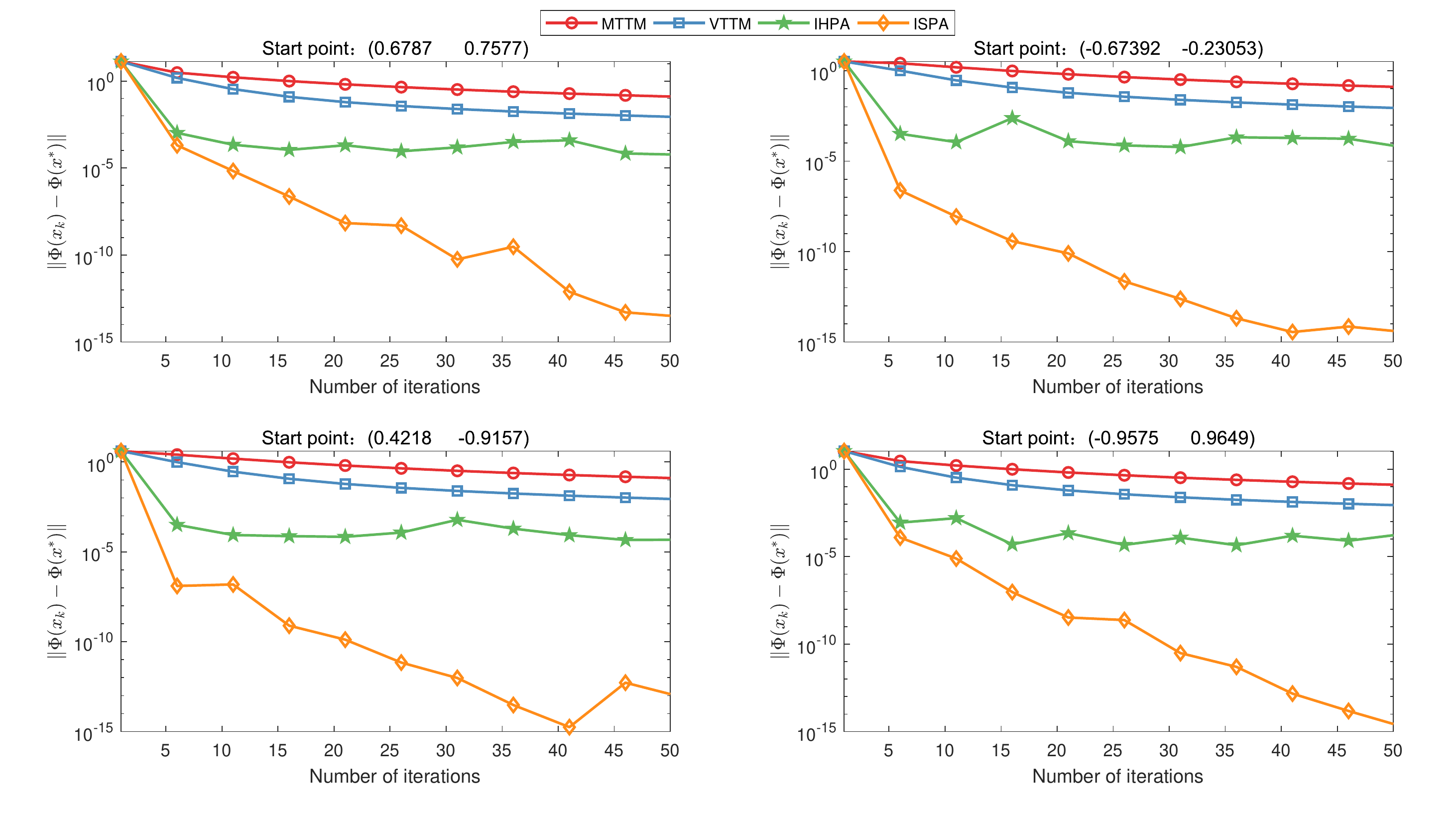}	
	\caption{Component convergence behavior of $\|\Phi(x_{n})-\Phi(x^{*})\| $ with different initials.}
	\label{figcom4}
\end{figure}

\section{Conclusion}\label{sec5}

Forward-Backward splitting algorithms are efficient  and powerful    to  monotone inclusion problems. In this paper, we investigated  the problem of  finding a zero of the sum of two monotone operators in real Hilbert spaces  by proposing two  projection-based algorithms with inertial effects.  Our algorithms  use  a new stepsizes rule which makes them more efficient and robust.

\end{document}